\documentclass[]{article}
\usepackage{amsmath}
\usepackage{amssymb}

\usepackage{authblk}

\usepackage{amsthm}
\date{}
\title{A note on constructive treatment of eigenvectors}

\author{P. Osinenko$^*$}
\author{G. Devadze}
\author{S. Streif}
\affil{Laboratory for Automatic Control and Systems Dynamics\\ Technische Universit{\"a}t Chemnitz }
\affil{$^*$\small{ pavel.osinenko@etit.tu-chemnitz.de} }

\newtheorem{lemma}{Lemma}
\newtheorem{theorem}{Theorem}

\newcommand{\lmultiset}{\{\!\vert}
\newcommand{\rmultiset}{\vert\!\}}

\usepackage{xspace}
\usepackage{xcolor}
\usepackage{mathtools}
\newcommand{\eg}{e.\,g.\xspace}
\newcommand{\ie}{i.\,e.\xspace}

\providecommand{\keywords}[1]{\textbf{\textit{Index terms---}} #1}

\begin{document}

\maketitle

\begin{abstract}
The eigenvalue problem plays a central role in linear algebra and its applications in control and optimization methods. In particular, many matrix decompositions rely upon computation of eigenvalue-eigenvector pairs, such as diagonal or Jordan normal forms. Unfortunately, numerical algorithms computing eigenvectors are prone to errors. Due to uncomputability of eigenpairs, perturbation theory and various regularization techniques only help if the matrix at hand possesses certain properties such as the absence of non--zero singular values, or the presence of a distinguishable gap between the large and small singular values. Posing such a requirement might be restrictive in some practical applications. In this note, we propose an alternative treatment of eigenvectors which is approximate and constructive. In comparison to classical eigenvectors whose computation is often prone to numerical instability, a constructive treatment allows addressing the computational uncertainty in a controlled way. 

\end{abstract}

\keywords{Eigenvalues, eigenvectors, constructive analysis, approximate solutions, fundamental theorem of algebra}

\section{Introduction}

Let $\mathbf{A}$ be a complex-valued $n \times n$ matrix. Its characteristic polynomial is given as:
\begin{equation}
	P_A(\lambda) = \det \left( \mathbf{A} - \lambda \mathbf{I} \right). \label{eq:char-poly} 
\end{equation}

An eigenpair $\left( \mathbf{v}_j, \lambda_j \right)$ is a root $\lambda_j, j \in \{ 1, \ldots n \}$ of \eqref{eq:char-poly} and a vector $\mathbf{v}_j$ satisfying:
\begin{equation}
	\mathbf{A} \mathbf{v}_j = \lambda_j \mathbf{v}_j. \label{eq:eig-vec} 
\end{equation}

Computing eigenpairs is crucial for many control and optimization methods because these methods are often based on certain matrix decompositions such as diagonal normal form, Jordan form or singular value decompositions (SVD). For example, SVD is used in $H_2$, $H_{\infty}$ and $\mu$-optimal controllers by decoupling the original system into lower-dimensional ones, \eg \cite{Hovd1997-SVD-ctrl}, and requires computing left and right eigenvectors. SVD is also used in deriving reduced models and controllers which preserve important system properties such as closed-loop stability, observability or controllability \cite{Sou2010-SVD-stab, Antoulas2005-appr-dyn-sys}. In model predictive control, SVD of the cost function's Hessian is often used (see, \eg \cite{Rojas2004-SVD-horizon-ctrl}). System identification is another field of application of SVD. For instance, Zhang et al.\,\cite{Zhang1994-SVD-Kalman} used SVD on extended Kalman filter to cope with numerical ill-conditioning. Eigenvectors are used in computing Jordan normal form for stability analysis of linear systems \cite{Rohrs1992-lin-ctrl-sys}. For a survey on controller design based on matrix normal forms refer to \cite{Walker2003-controllers}.

Computing eigenvectors amounts to finding non--trivial solutions to the system of linear equations \eqref{eq:eig-vec}. The classical method of solving systems of linear equations is Gaussian elimination (see, \eg, \cite{Lay2012-lin-alg}).  In general, solving a system of linear equations $\mathbf{A}\mathbf{v} = \mathbf{b}$, where $\mathbf{A}$ is the coefficient matrix, $\mathbf{b}$ is the data vector, and $\mathbf{v}$ is an unknown vector, amounts to comparing real numbers to zero. In computations, real numbers are represented by computer programs or algorithms that compute the respective approximations. There is, however, no algorithm that can decide whether $\alpha = \beta$ or $\alpha \neq \beta$ for arbitrary real numbers $\alpha$ and $\beta$. Such an algorithm would be equivalent to solving the problem of deciding whether an arbitrary computer program terminates or not -- which is impossible as shown by Turing \cite{Turing1936-Halting-problem}.
When performing elementary matrix row operations, one might have to deal with numbers that are small relative to the machine precision without actually knowing whether they are actually zero or not.
In practice, small numbers are either set to zero or to machine epsilon. However, this may lead to a large deviation of an approximate solution from the true one \cite{Neumaier1998-ill-posed-problems}.

Besides the problem of approximating small numbers when solving $\mathbf{A}\mathbf{v} = \mathbf{b}$, a typical problem is that the coefficient matrix $\mathbf A$ is singular. A comprehensive analysis of these ill-posed problems was carried out within perturbation theory \cite{Chandrasekaran1991-pertur-thr,Tikhonov1995-regularization,Neumaier1998-ill-posed-problems,Hansen1998-ill-posed-problems,Engl1996-regulariz}. The techniques	that cope with solving such ill-posed problems approximately, mainly use regularization or truncated singular value decomposition (TSVD). In brief, if the coefficient matrix $\mathbf{A}$ of a problem $\mathbf{A}\mathbf{v} = \mathbf{b}$ is close to singular, it is substituted by some other matrix which is regular. For instance, Tikhonov regularization \cite{Tikhonov1995-regularization} suggests to use $\left( \mathbf{A} \mathbf{A}^* + h^2 \mathbf{I}\right)^{-1} \mathbf{A}^* \mathbf{b}$ as an approximate solution to $\mathbf{A}\mathbf{v} = \mathbf{b}$. In TSVD, $\mathbf{A}$ is decomposed into $\mathbf{U} \mathbf{\Sigma} \mathbf{V}^*$, where $\mathbf{U}$ and $\mathbf{V}^*$ are unitary matrices, and $\mathbf{\Sigma}$ is a diagonal matrix of singular values $\{ \sigma_j \}_j$ of $\mathbf{A}$, \ie, square roots of the eigenvalues of $\mathbf{A}^*\mathbf{A}$. The singular values, which are small relative to the given threshold, are set to zero (``truncated''). The resulting approximate solution to $\mathbf{A}\mathbf{v} = \mathbf{b}$ may be found as follows:
\begin{equation}
	\hat{\mathbf{v}} = \sum_{j \in I_1} \frac{ \mathbf{u}_j^{T} \left( \mathbf{A}^T \mathbf{b} \right) }{ \sigma_j } \mathbf{v}_j +\sum_{  j \in I_2 } t_j \mathbf{v}_j, \label{eq:sol-by-TSVD} 
\end{equation}
where $\mathbf{u}_{j}, \mathbf{v}_{j}$ are the $j$--th column of $\mathbf{U}$ and $\mathbf{V}$, respectively, with index sets $I_1 \coloneqq \{ j : \sigma_j \neq 0 \}$ and $I_2 \coloneqq \{ j : \sigma_j = 0 \}$, and $t_j$ are free variables associated with those singular values that have been truncated.

In practice, such regularization techniques work well only when there is a distinguishable gap between the singular values that are large (\ie\,contained in $I_1$) and those which are near zero (\ie\,contained in $I_2$) relative to the machine precision. In general, any approximate solution of this kind may yield large approximation errors. The bound on a solution is proportional either to the so--called condition number $\| \mathbf{A} ^{-1} \| \| \mathbf{A} \| $, or, equivalently, to the reciprocal of the smallest singular value of $\mathbf{A}$, or to the said gap between singular values. 
Here 
%
	$\| \mathbf{A} \| \triangleq \sup \{ \| \mathbf{A} \mathbf{v} \|: \mathbf{v} \in \mathbb{C}^n \land \| \mathbf{v} \| = 1 \}$ 
denotes the operator norm of $\mathbf{A}$ and $\| \cdot \|$ is the Euclidean vector norm.
Then, if $\mathbf{A}$ has actually a zero singular value or if the said gap is close to the machine epsilon, regularization techniques may yield poor approximate solutions \cite{Neumaier1998-ill-posed-problems}.

When looking numerically for non--trivial solutions to $\mathbf{A}\mathbf{v}=\mathbf{0}$, one often deals with some approximation $\hat{\mathbf{A}}$ to $\mathbf{A}$. Let $\hat{\mathbf{A}} = \mathbf{A} + \mathbf{E}$, where $\mathbf{E}$ is some perturbation matrix. Eigenvectors of $\hat{\mathbf{A}}$ then satisfy:
\begin{equation}
	\hat{\mathbf{v}}_j = \mathbf{v}_j  + \sum_{\substack{k=1 \\ k \neq j}} \frac{ \mathbf{v}_k^T \mathbf{E} \mathbf{v}_k }{ \lambda_j - \lambda_k } \mathbf{v}_j, \label{eq:perturbed-eig-vec} 
\end{equation}
where $\{ \lambda_j \}_j$ and $\{ \mathbf{v}_j \}_j$ are the original eigenvalues and eigenvectors. For details and numerical algorithms, refer to \cite{Trefethen1997-num-lin-alg}. Some other algorithms, such as approximation of eigenvectors by least squares \cite{Hecker2007-appr-eig-vec-LS}, rely on the knowledge of the first non--zero singular value of $\mathbf{A}$. Therefore, solving systems of linear equations $\mathbf{A}\mathbf{v} = \mathbf{b}$, and finding eigenvectors in particular, requires certain properties of the original problem to be satisfied which might be restrictive in some applications.

To approach the problem from a computational point of view, linear algebra was also addressed in computable analysis as developed by Weihrauch \cite{Weihrauch2012-comp-analys}. Computable analysis investigates mathematical objects that can be computed in principle, \ie, by an algorithm that is guaranteed to yield a correct answer. Brattka et al. \cite{Brattka2002-comp-lin-eqns} investigated the problem of computability of solutions to a problem $\mathbf{A}\mathbf{v} = \mathbf{b}$. They concluded that the solutions are computable provided that the rank of $\mathbf{A}$ is known in advance. In general, however, the rank is uncomputable \cite{Ziegler2000-comp-dim-lin-subspc}. 
There  also exist constructive frameworks of algebra. For a comprehensive course, please refer to a canonical book on the subject by Mines et al. \cite{Mines1988-constr-alg}. The eigenvalue problem was investigated in constructive mathematics for particular linear operator decompositions. For example, Ye \cite{Ye2011-SF} investigated the spectral decomposition using constructive functional analysis. He used an assumption that all the eigenvalues be given with their multiplicities known beforehand. With this assumption, he was able to exactly reconstruct the spectral theorem.

In this note, we are neither concerned with \textit{exact} eigenvectors, nor do we seek for new bounds on approximate solutions in the spirit of perturbation theory. Instead, we suggest to treat eigenvectors from the standpoint of constructive mathematics, as developed by Bishop et al. \cite{Bishop1985-constr-analysis}. The advantage of this approach is that the derived proofs may be translated into effective algorithms that always output a correct result \cite{Bishop1985-constr-analysis}. This, in our view, is especially important in applications such as control engineering. We do not look for solutions to \eqref{eq:eig-vec}, but propose to use the following notion of an approximate eigenvector:
\begin{equation}
	\| \mathbf{A} \hat{\mathbf{v}} - \hat{\lambda}_j \hat{\mathbf{v}} \| \leq \varepsilon, \label{eq:eps-eig-vec} 
\end{equation}
where $\varepsilon$ is the predefined accuracy, $\hat{\mathbf{v}}$ is an $\varepsilon$--eigenvector with $\| \hat{\mathbf{v}} \| = 1$, and $\hat{\lambda}_j$ is an approximate eigenvalue of $\mathbf{A}$ in the sense that $\det ( \mathbf{A} - \hat{\lambda}_j \mathbf{I} ) \le \varepsilon$. We constructively prove existence of $\varepsilon$--eigenvectors using the algebraic closure of rational complex numbers and continuity of roots of polynomials. The results derived in the present note do not require a prior knowledge of constructive functional analysis or measure theory. First, we discuss certain important notions of the constructive analysis in the preliminaries section, then we  constructively verify that roots of polynomials are continuous in coefficients, and finally we prove existence of $\varepsilon$--eigenvectors.
We conclude by pointing out the importance and usages of $\varepsilon$--eigenvectors in control and optimization.

\section{Preliminaries}

The aim of this section is to briefly introduce certain key notions of constructive mathematics that will be needed in this note. For a comprehensive description, refer to \cite{Bishop1985-constr-analysis,Bridges2007-techniques,Bridges1987-varieties}. The central notion is an \textbf{operation} which is an algorithm that produces a unique result in a finite number of steps for each input from its domain. For example, a \textbf{real} number $x$ is a \emph{regular} Cauchy sequence of rational numbers in the sense that 
\begin{equation*}
	\forall n, m \in \mathbb{N} \left[ |x(n)-x(m)|\leq\frac{1}{n}+\frac{1}{m} \right]
\end{equation*}
where $x(n)$ is an \emph{operation} that produces the $n-$th rational approximation to $x$. A \textbf{set} is a pair of operations: $\in$ determines that a given object is a member of the set, and $=$ determines whether two given set members are equal. Existence and universal quantifiers are interpreted as follows: $\exists x\in A \left[ \varphi (x) \right]$ means that an operation was derived that constructs an instance $x$ along with a proof of $x\in A$ and a proof of the logical formula $\varphi (x)$ as \emph{witnesses}; $\forall x\in A \left[ \varphi (x) \right]$ means that an operation was derived that proves $\varphi (x)$ for any $x$ provided with a witness for $x\in A$. A \textbf{multiset} is a finite collection of objects that allows repeated elements. We denote multisets by $\lmultiset \cdot \rmultiset$. For example $\lmultiset 1, 1 \rmultiset$ is a multiset. A function $f:[a ,b] \rightarrow \mathbb{R}$ is \textbf{uniformly continuous} if there exists an operation $\omega: \mathbb{Q} \rightarrow \mathbb{Q}$ called \textbf{modulus of continuity} such that $\forall \varepsilon \in \mathbb{Q} \forall x, y \in [a, b] \left[ |x - y| \le \omega(\varepsilon) \implies |f(x) - f(y)| \le \varepsilon \right]$.

A \textbf{complex number} $x + y i$  is identified with a pair $( x , y )$ of real numbers. The set of complex numbers $\mathbb{C}$ is a metric space with the metric $\rho(z, w) \triangleq |z - w| $. A \textbf{sphere} centered at $z$ with the radius $r>0$ is denoted by $\mathcal{S}(z, r)$. Let $K$ respectively $U$ be a compact respectively any set in $\mathcal{C}$ respectively, and let $r > 0$.  We define
\begin{equation*}
 K_r \triangleq \{z  \in \mathbb{C}: \rho(z,K) \le r\} 
\end{equation*}

The set $K$ is well-contained in $U$ if there exists $r > 0$ such that $K_r \subseteq U$. A \textbf{path} $\gamma$ is a uniformly continuous function $\gamma: [a,b] \rightarrow  \mathbb{C}$ such that:
\begin{equation*}
	\exists t_0, \ldots t_n, n > 0 \left[ t_0 = a \le t_1 \le \dots \le t_n \le b \right],
\end{equation*}
and $\gamma$ is differentiable on each interval $[t_i, t_{i+1}] , i=0, \dots , n-1$. A path is closed if $\gamma(a) = \gamma(b)$. By $\text{car}(\gamma)$, we denote the closure of the range of $\gamma$. Notice that $\text{car}(\gamma)$ is compact and $\rho(z, \text{car}(\gamma))$ exists for any $z$ since $\gamma$ is uniformly continuous. If a function $f$ is differentiable on $\text{car}(\gamma)$ and $\gamma$ is closed, then
\begin{equation*}
	\int_{\gamma} f' dz =0.
\end{equation*}

If $\gamma$ is a closed path, $z_0 \in \mathbb{C}$, and $ \rho(z_0, \text{car}(\gamma) ) > 0 $ then:
\begin{equation*}
	\int_{\gamma} \cfrac{dz}{z-z_0} dz = 2 \pi i \mathcal{W}(\gamma, z_0).
\end{equation*}
for some integer $\mathcal{W}(\gamma, z_0)$ called the \textbf{winding number} of $\gamma$ around $z_0$. We will denote the number of zeros of a function $f$ within a closed path $\gamma$ by $\mathcal{N}(f, \gamma)$. The \textbf{length} of a path $\gamma$ is denoted by $\text{len } \gamma$ and defined by $\int_{\gamma} 1 dz$. A function $P(z) = \sum_{k=0}^{n} a_k z^k $ is a \textbf{polynomial} of degree $n$ if $a_n \ne 0$. It is \textbf{monic} if $a_n=1$. A continuous function $f$ on an open set $U$ is \textbf{analytic} on $U$, if $ \int_{\gamma} f dz = 0 $ for any triangular path $\gamma$ going through some points $z_1, z_2, z_3$ such that the convex hull of $z_1, z_2, z_3$ is a subset of $U$. In the next section, we present necessary lemmas, required for the main theorem.

\section{Technical lemmas}
In this section, we show that continuity of roots of polynomials admits a constructive proof. Most of the classical derivations are easily transferable into the constructive setup.

Let $P(z)$ be a monic polynomial of degree $n$ over $\mathbb{C}$. By the constructive Fundamental Theorem of Algebra (FTA) \cite[p.~156]{Bishop1985-constr-analysis}, it has $n$ roots. Therefore,
\begin{equation*}
	P(z) = z^n + \sum_{k=0}^{n-1}a_k z^k = \prod_{k=1}^{n} (z - z_k)
\end{equation*}

Let $\hat{P}(z)$ be some monic polynomial of degree $n$ whose coefficients $\hat{a}_k$ approximate those of $P(z)$ in the sense that $\forall k \left[ |a_k - \hat{a}_k | \le \varepsilon \right]$. The polynomial $\hat{P}(z)$ may be decomposed as well:
\begin{equation}
	\hat{P}(z) = z^n + \sum_{k=0}^{n-1}\hat{a}_kz^k = \prod_{k=1}^{n} (z - \hat{z}_k)
\end{equation}

We look for a bound $\delta$, that depends on $\varepsilon$ such that $\forall k \left[ |z_k - \hat{z}_k| \le \delta(\varepsilon) \right]$. That is, if the coefficients of $\hat{P}(z)$ are close to the coefficients of $P(z)$, then the respective roots are close as well. To show this, we require several technical lemmas. The first is:
\begin{lemma} \label{lemma-arg-principle-poly}	
	Let $P(z)$ be a polynomial of degree $n$, and let $\bar{\mathcal{B}}$ be some closed ball well-contained in an open set $U \subset \mathbb{C}$ with the boundary $\mathcal{S}$. Let $\gamma$ denote a closed path such that $\text{ car }(\gamma) = \mathcal{S}$. Assume that $P(z)$ has no roots on  $\text{ car }(\gamma)$. It follows that:
	\begin{equation*}
		\frac{1}{2 \pi i} \oint_{\gamma} \frac{P'(z)}{P(z)} dz = \mathcal{N}(P, \gamma)
	\end{equation*}	
\end{lemma}

\begin{proof}
    By the FTA, $P(z)$ has a multiset of $n$ roots $\lmultiset z_1, \ldots z_n \rmultiset$. Without loss of generality, we may assume that $P(z)$ is monic. Then, it can be expressed as $P(z) = (z - z_1) P_1(z)$ where $P_1(z) = \prod_{k=2}^n(z - z_k)$. Therefore, $P'(z) = P_1(z) + (z - z_1)P'_1(z)$. It follows that:
    \begin{equation*}
    \begin{split}
        \frac{1}{2 \pi i} \oint_{\gamma} \cfrac{P'(z)}{P(z)} dz & = \frac{1}{2 \pi i} \oint_{\gamma} \left( \cfrac{1}{(z-z_1)} + \cfrac{P'_1(z)}{P_1} \right) dz \\ 
        & = \frac{1}{2 \pi i} \oint_{\gamma} \cfrac{1}{(z-z_1)} + \frac{1}{2 \pi i} \oint_{\gamma} \cfrac{P'_1(z)}{P_1} dz.
    \end{split}
    \end{equation*}    
    
    The first integral equals $\mathcal{W}(\gamma, z_1) = 1$ if $z_1$ is in $\bar{\mathcal{B}}$, and $0$ otherwise. Now, proceeding by finite induction, we obtain the result.
\end{proof}

The next lemma shows that a parametric integral over a path is uniformly continuous.
\begin{lemma} \label{par-integral-cont}
	Let $f(t,z)$ be a uniformly continuous function on $[t_0, t_1] \times U$, where $t_0 < t_1, t_0, t_1 \in \mathbb{R}$, and $U$ is an open set in $\mathbb{C}$. Let $\gamma: [0, 1] \rightarrow \mathbb{C}$ be a path in $U$. Define $I(t) := \int_{\gamma} f(t,z) dz $. It follows that $I(t)$ is a uniformly continuous function of $t$.
\end{lemma}

\begin{proof}
    Let $\tau_1, \tau_2, \tau_1 \le \tau_2$ be two points in $[t_0, t_1]$. First, observe that:
    \begin{equation*}
        | I(\tau_1) - I(\tau_2) | = \left | \int_{\gamma} ( f(\tau_1, z) - f(\tau_2, z) ) dz \right | \le \int_{\gamma} \left |  f(\tau_1, z) - f(\tau_2, z) \right | dz. 
    \end{equation*}
    
    Let $\omega_f$ be the modulus of continuity of $f$. It follows that if $ | \tau_1 - \tau_2 | \le \omega_f( \varepsilon )$ for some $\varepsilon > 0$, then $ |f(\tau_1, z) - f(\tau_2, z)| \le \varepsilon $. Therefore,
    \begin{equation*}
        \int_{\gamma} \left |  f(\tau_1, z) - f(\tau_2, z) \right | dz \le \int_{\gamma} \varepsilon dz = \varepsilon \mbox{len}(\gamma). 
    \end{equation*}    
    
    The result follows by setting $\omega_I(\varepsilon) := \omega_f \left( \frac{ \varepsilon } { \mbox{len}(\gamma) } \right)$.
\end{proof}

We require a variant of the Rouche's Theorem \cite[p.~131]{Beardon1979-compl-analys} for polynomials:
\begin{lemma} \label{lemma-Rouche-poly}
    Let $P(z)$ and $Q(z)$ be polynomials of degree $n$, and let $\bar{\mathcal{B}}$ be some closed ball in $\mathbb{C}$ with the boundary $\mathcal{S}$. Let $\gamma$ denote a closed path such that $\text{ car }(\gamma) = \mathcal{S}$. Assume that $P(z)$ has no roots on  $\text{ car }(\gamma)$. If $|Q(z)| < |P(z)|$ for all $z$ in $\bar{\mathcal{B}}$, then $ \mathcal{N} (P, \mathcal{S}) = \mathcal{N} (P + Q, \mathcal{S}) $.
\end{lemma}

\begin{proof}
    Define the following function: 
    \begin{equation*}
        I(t):= \frac{1}{2 \pi i} \oint_{\gamma} \cfrac{ P'(z) + t Q'(z) }{ P(z) + t Q(z) }, t \in [0, 1].
    \end{equation*}
    Since $P(z)$ has no roots on  $\text{ car }(\gamma)$, $P(z) + t Q(z)$ also has no roots on $\text{ car }(\gamma)$. By Lemma \ref{lemma-arg-principle-poly}, $I(t) = \mathcal{N} (P + t Q, \mathcal{S}) \in \mathbb{N}$. The condition $|Q(z)| < |P(z)|$ on $\bar{\mathcal{B}}$ implies that for any $z$ in $\bar{\mathcal{B}}$, there exists a positive rational complex number $p$ such that:
    \begin{equation*}
        |Q(z)| < p < |P(z)|.
    \end{equation*}
    
    By the density of rational complex numbers in $\mathbb{C}$, there exists also a positive $q$ such that the following inequality holds:
    \begin{equation*}
        |Q(z)| < p < q < |P(z)|.
    \end{equation*}    
    
    By the triangle inequality, we have:
    \begin{equation*}
        |P(z) + t Q(z)| > |P(z)| - t |Q(z)| > |P(z)| - |Q(z)| > q - p.
    \end{equation*}     
    It follows that $\frac{1} { |P(z) + t Q(z)| } $ is analytic on $\text{car}(\gamma)$. Hence, $\frac{P'(z) + t Q'(z) }{ P(z) + t Q(z) }$ is analytic on $\text{car}(\gamma)$ as well. By Lemma \ref{par-integral-cont}, $I(t)$ is uniformly continuous. But since $I(t)$ is an integer for all $t$, $I(t)$ is a constant function. Therefore, $I(0) = I(1)$ which means that $ \mathcal{N} (P, \mathcal{S}) = \mathcal{N} (P + Q, \mathcal{S}) $.
\end{proof}

We are ready to prove the following lemma:
\begin{lemma} \label{lemma-continuity-rts}
	Let $ P(z) = z^n + \sum_{k=0}^{n-1}a_k z^k $ be a monic polynomial. For any $\varepsilon > 0$, there exists a monic polynomial $\hat{P}(z) = z^n + \sum_{k=0}^{n-1} \hat{a}_k z^k$ and $\delta > 0$, that depends on $\varepsilon$, such that $\forall k \left[ |z_k - \hat{z}_k | \le \varepsilon \right]$, where $z_k, \hat{z}_k$ are a root of $P(z)$ and a root of $\hat{P}(z)$ respectively, and $\forall k \left[ |a_k - \hat{a}_k | \le \delta( \varepsilon ) \right]$.
\end{lemma}

\begin{proof}
	By the FTA, $P(z)$ has a multiset of $n$ roots $\lmultiset z_1, \ldots z_n \rmultiset$. By approximating $z_1, \ldots z_n$ sufficiently, we can find some $m$ disjoint closed balls $\{ \bar{\mathcal{B}}(c_j, n \varepsilon) \}_j$ of radius $n \varepsilon$ each containing some $m_j$ points of $\lmultiset z_1, \ldots z_n \rmultiset$, and such that no root is on the boundary of the respective ball. 
		Let 
	\begin{equation*}
		s_j: = \inf_{\mathcal{S}(c_j, n\varepsilon)} |P(z)|, j=1, \ldots m.
	\end{equation*}
	
	Notice that infima and suprema of polynomials on $\mathcal{S}(c_j, n \varepsilon)$ exist since $\mathcal{S}(c_j, n \varepsilon)$ are compact sets. Further, let
	\begin{equation*}
		p_j:= \sup_{\mathcal{S}(c_j, n \varepsilon)}  \left\{ 1 + \sum_{k=1}^{n-1}|z^{n-k}| \, \right\} = \sup_{\mathcal{S}(c_j, n \varepsilon)}  \left\{1+|z^{n-1}| + |z^{n-1}|+\dots + |z|\, \right\}.
	\end{equation*}	
	
	Let $\delta$ be such that $\forall j = 1, \dots m \left[ \delta < \frac{s_j}{p_j} \right]$ and $\forall k = 1, \ldots n-1 \left[ |a_k - \hat{a}_k | \le \delta \right]$. Then, it follows that if $z \in \{ \bar{\mathcal{B}}(c_j, n \varepsilon) \}_j$:
	\begin{equation*}
	    \begin{split}
	        |\hat{P}(z) - P(z)| = \left |\sum_{k=0}^{n-1}(\hat{a}_k - a_k)z^k\right | \le \sum_{k=0}^{n-1}|\hat{a}_k - a_k||z^k| \le \\
		\delta \sum_{k=0}^{n-1} |z^k| \le \delta p_j < s_j \le |P(z)|
	    \end{split}
	\end{equation*}
	Therefore, $ \forall j = 1, \dots m \left[ z \in \bar{\mathcal{B}}(c_j, n \varepsilon) \implies |P(z)| > | \hat{P}(z) - P(z) | \right]$. By Lemma \ref{lemma-Rouche-poly}, $ \mathcal{N}(P, \mathcal{S}(c_j, n \varepsilon)) = \mathcal{N}(\hat{P}, \mathcal{S}(c_j, n \varepsilon)) = m_j$ since $(\hat{P}(z) - P(z)) + P(z) = \hat{P}(z)$. It follows that for a given $\varepsilon$, there is a $\delta$ that establishes the desired property. Notice that by the density of rational complex numbers, there exists such a polynomial $\hat{P}(z)$ with rational coefficients.
\end{proof}

Using the lemmas in this section, and algebraic closure of rational complex numbers, we are ready to prove the existence of approximate eigenvectors.

\section{Constructive $\varepsilon$--eigenvectors and discussion}

\begin{theorem} \label{thm-eps-eig-vec}
    Let $A$ be a complex-valued $n \times n$ matrix with the characteristic polynomial $P_A(\lambda)$. For any $\varepsilon > 0$, there exist some $m$ linearly independent vectors $\hat{\mathbf{v}}_1, \ldots \hat{\mathbf{v}}_m$ and complex numbers $\hat{\lambda}_1, \ldots \hat{\lambda}_m$ such that:
    \begin{equation*}
	    \forall k = 1, \ldots m  \left [ \| \mathbf{A} \hat{\mathbf{v}}_k - \hat{\lambda}_k \hat{\mathbf{v}}_k \| \leq \varepsilon \right ],
    \end{equation*}   
where $ \forall k = 1, \ldots m$  $\hat{\lambda}_k $ is an approximate eigenvalue of $A$.
\end{theorem}

\begin{proof}
    Let $P_A(\lambda) = \det \left( \mathbf{A} - \lambda \mathbf{I} \right)$ be the characteristic polynomial of $A$. Notice that $P_A$ is monic of degree $n$ and thus, by the Fundamental Theorem of Algebra, there exists a multiset of $n$ roots $\lmultiset \lambda_1, \ldots \lambda_n \rmultiset$. Let $\hat{\mathbf{A}} = \mathbf{A}_N$ be the $N-$th rational approximation to $\mathbf{A}$. That is, for each entry $a_{jk}$ of $A$, it follows that $| a_{jk} - \hat{a}_{jk} | \le \frac{2}{N}$ where $\hat{a}_{jk}$ is the respective entry of $\hat{\mathbf{A}}$. Since the coefficients of the characteristic polynomial are defined via the finite product and sum of the matrix entries, we can find such an $N$ that, by using Lemma \ref{lemma-continuity-rts}, the roots $\hat{\lambda}_1, \ldots \hat{\lambda}_n$ of $\hat{\mathbf{A}}$ will satisfy $\forall k = 1, \ldots n \left[ | \hat{\lambda}_k - \lambda_k | \le \varepsilon \right]$. Notice that the entries of $\hat{\mathbf{A}}$ are rational complex numbers. Thus, $\hat{\lambda}_1, \ldots \hat{\lambda}_n$ are algebraic numbers. By Lemma 4.1 from \cite[p.~8]{Beeson1980}, we can decide whether $\hat{\lambda}_j = \hat{\lambda}_k$ or $\hat{\lambda}_j \ne \hat{\lambda}_k$ for $j \ne k, j,k = 1, \ldots n$. Suppose that there are $m$ distinct numbers in $\lmultiset \hat{\lambda}_1, \ldots \hat{\lambda}_n \rmultiset$, and their algebraic multiplicities are $l_1, \ldots l_m$. Since
    \begin{equation*}
    	\mathbf{\hat{A}} \mathbf{\hat{v}} = \hat{\lambda}_k \mathbf{\hat{v}}, k = 1, \ldots m  
    \end{equation*}
are systems of linear equations with coefficients being algebraic numbers, we can solve them exactly for $\mathbf{\hat{v}}$ by Gaussian elimination. Notice that these systems of linear equations are consistent, and $\mathbf{\hat{A}} - \hat{\lambda}_k I$ is singular, so there is at least a one-dimensional subspaces of vectors that satisfy the respective system of equations. Therefore, we may assume that $\| \mathbf{\hat{v}} \| = 1$. Notice that a solution  $\mathbf{\hat{v}}$ is a tuple of algebraic numbers. Further, for each $\hat{\lambda}_k$, we can determine the geometric multiplicity, again, by Gaussian elimination on $\mathbf{\hat{A}} - \hat{\lambda}_k I$. If the geometric multiplicity $g_k$ of some $\hat{\lambda_k}$ is greater than one, we can find linearly independent vectors $\mathbf{\hat{v}}_k^{1} \ldots \mathbf{\hat{v}}_k^{g_k}$ all satisfying $\mathbf{\hat{A}} \mathbf{\hat{v}} = \hat{\lambda}_k \mathbf{\hat{v}}$. Without loss of generality, we may assume that all geometric multiplicities are equal to one. At this point, we have a finite set of eigenpairs $( \hat{\lambda_1}, \mathbf{\hat{v}}_1 ), \ldots ( \hat{\lambda}_m, \mathbf{\hat{v}}_m )$ of $\mathbf{\hat{A}}$. For $k = 1, \ldots m$, it follows that:
\begin{equation*}
     \begin{split}
         \| \mathbf{A} \mathbf{\hat{v}}_k -  \hat{ \lambda }_k \mathbf{\hat{v}}_k  \| & = \| \mathbf{A} \mathbf{\hat{v}}_k  +  \mathbf{\hat{A}}  \mathbf{\hat{v}}_k  -  \mathbf{\hat{A}}  \mathbf{\hat{v}}_k  -  \hat{\lambda}_k  \mathbf{\hat{v}}_k  \| \\
        & = \| ( \mathbf{A} - \mathbf{\hat{A}} ) \mathbf{\hat{v}}_k + 0 \| \\
        & = \| ( \mathbf{A} - \mathbf{\hat{A}} ) \mathbf{\hat{v}}_k \| \\
        & \le \| ( \mathbf{A} - \mathbf{\hat{A}} ) \| \| \mathbf{\hat{v}}_k \| \\
        & = \| ( \mathbf{A} - \mathbf{\hat{A}} ) \| \le n \sqrt{n} \varepsilon
     \end{split}
\end{equation*}

    By setting $N$ accordingly, the result follows.
\end{proof}

Note that we cannot prove constructively that the approximate eigenvectors converge to the classical eigenvectors, which implies that we cannot prove existence of eigenvectors in the general case. We now prove another useful inequality. Let $\lambda_j$ be such that $| \lambda_j - \hat{ \lambda }_k | \le \varepsilon$ for some $k$. It follows that:
\begin{equation*}
 \begin{split}
     \| \mathbf{A} \mathbf{\hat{v}}_k -  \lambda_j \mathbf{\hat{v}}_k  \| & =
	\| \mathbf{A} \mathbf{\hat{v}}_k +  \mathbf{\hat{A}}  \mathbf{\hat{v}}_k  -  \mathbf{\hat{A}}  \mathbf{\hat{v}}_k  + \hat{\lambda}_k  \mathbf{\hat{v}}_k -\hat{\lambda}_k  \mathbf{\hat{v}}_k - \lambda_j \mathbf{\hat{v}}_k  \| \\
	& = 	\| \mathbf{A} \mathbf{\hat{v}}_k  -  \mathbf{\hat{A}}  \mathbf{\hat{v}}_k  + \hat{\lambda}_k  \mathbf{\hat{v}}_k - \lambda_j \mathbf{\hat{v}}_k  \| \\
	& \le \| (\mathbf{A}  -  \mathbf{\hat{A}} ) \mathbf{\hat{v}}_k  \| + \| ( \hat{\lambda}_k  - \lambda_j ) \mathbf{\hat{v}}_k  \| \\
	& \le \| \mathbf{A}  -  \mathbf{\hat{A}}  \| \| \mathbf{\hat{v}}_k  \| + \|  \hat{\lambda}_k  - \lambda_j  \| \| \mathbf{\hat{v}}_k  \| \\
	& \le (1 + n \sqrt(n)) \varepsilon
 \end{split}
\end{equation*}

There exists a concept of an $\varepsilon$--eigenvector, introduced by \cite[p.~372]{Bishop1985-constr-analysis},  in the following sense: if $A$ is an operator on a Hilbert space $H$, then a vector $v \in H$ is an $\varepsilon$--eigenvector of $A$ if $\| x \| = 1$ and $ \| A v - (A v, v) v \| \le \varepsilon$, where $(A v, v)$ is the dot product. Observe that in our case, $\mathbf{\hat{A}} \mathbf{\hat{v}}_k = \hat{\lambda}_k \mathbf{\hat{v}}_k$. Transposing both sides and multiplying by $\mathbf{\hat{v}}$, we get:
\begin{equation*}
    (\mathbf{\hat{A}} \mathbf{\hat{v}}_k)^T  \mathbf{\hat{v}}_k = \hat{\lambda}_k \mathbf{\hat{v}}_k^T \mathbf{\hat{v}}_k = \hat{\lambda}_k \| \mathbf{\hat{v}}_k \| = \hat{\lambda}_k.  
\end{equation*}

Therefore, $\| \mathbf{\hat{A}} \mathbf{\hat{v}}_k - (\mathbf{\hat{A}} \mathbf{\hat{v}}_k)^T \mathbf{\hat{v}}_k \| \le n \sqrt{n} \varepsilon$ which means that $\mathbf{\hat{v}}_k$ is an $n \sqrt{n} \varepsilon$--eigenvector of $\mathbf{A}$ in the sense of \cite{Bishop1985-constr-analysis}.

\section{Conclusion and Discussion}

In this note, we studied the computation of eigenvectors from the standpoint of constructive mathematics.
The derived results allow the user to define the bound on the uncertainty $\varepsilon$ of the eigenvectors. The computational content of the Theorem~\ref{thm-eps-eig-vec} is an algorithm that computes a set of unit vectors, \ie\ the $\varepsilon$--eigenvectors. 
The algorithm always outputs a correct result, whereas computation in the context of the exact eigenvectors may suffer from numerical instability. 

In controller design, one can guarantee robustness of computation by giving up the concept of the exact eigenvectors and focusing on $\varepsilon$--eigenvectors instead. When developing a controller that relies on the standard notion of exact eigenvectors (see, \eg, \cite{Hovd1997-SVD-ctrl}), the user may need to manually verify that the outcome of the computation is plausible. Then, if the user suspects a numerical issue, the procedure may need to be repeated using other initial guess or another solution algorithm. In contrast, the advantage of the constructive treatment of eigenvectors and matrix decompositions lies in the fact that the respective constructions are effective, \ie they always and automatically guarantee a correct result in terms of the given specification or predefined precision $\varepsilon$.

In practical applications, an essentially approximate treatment may be considered as a suitable substitute for exact eigenvectors where guaranteeing a predefined precision is essential. 
The derived construction meets the user specification of computation accuracy and might be useful in deriving various matrix normal forms and decompositions in an essentially approximate format. This may be especially important in applications with strict requirements for numerical stability, such as feedback control. Moreover, as formal verification methods gain popularity in control \cite{Bernardeschi2016-proof-assist-ctrl,Araiza-Illan2014-formal-ver-ctrl,Damm07-automating-verification,Gao2014-desciptive-ctrl}, effective computations and formal proofs become more important. We therefore assume that constructive mathematics may provide a coherent answer and appropriate means to assess uncertainty of software implementation of control systems. 

\section*{Acknowledgment}
The authors would like to thank D. Bridges for valuable conversations on the subject of the present work.


\bibliographystyle{unsrt}

\end{document}